\documentclass[10pt]{article}

\usepackage{amsmath, amscd, amssymb, amsthm}
\usepackage{pdfsync}

\newcommand{\R}{\mathbb{R}}

\newcommand{\C}{\mathbb{C}}
\newcommand{\bC}{\mathbb{C}}

\newcommand{\del}{\partial}
\newcommand{\real}{\textrm{Re }}

\newcommand{\pd}[2]{\frac{\partial #1}{\partial #2}}

\newcommand{\sumj}{\sum_{j=1}^{n-1}}

\newcommand{\evec}{\ensuremath{ \vec{\mathbf e }}}

\newtheorem{theorem}{Theorem}[section]
\newtheorem{lemma}[theorem]{Lemma}
\newtheorem{proposition}[theorem]{Proposition}
\newtheorem{corollary}[theorem]{Corollary}
\theoremstyle{definition}

\begin{document}

\title{Extension of Solutions to Holomorphic Partial Differential Equations}
\author{Jonathan Armel, Peter Ebenfelt}
\maketitle

\begin{abstract}
We derive sufficient conditions on a strictly pseudoconvex domain $\Omega \subset \C^n$ and a linear, holomorphic partial differential operator $P(z,\pd{}{z})$ for the existence of solutions to $Pu = 0$ which are holomorphic in $\Omega$ near $p\in\del\Omega$, but cannot be prolonged holomorphically across $p$. We first show how, given the existence of a supporting everywhere characteristic analytic hypersurface, a theorem of Tsuno can be used to construct the desired solutions. We then derive necessary conditions for the existence of such a hypersurface under the assumption that $\Omega$ is strictly pseudoconvex and simply characteristic at $p\in\del\Omega$ with respect to $P$.
\end{abstract}

\section{Preliminaries}

The two primary objects of study in this paper are a domain $\Omega$ in $\C^n$ with smooth boundary near a distinguished point $p \in \del \Omega$, and a linear partial  differential operator (PDO)
\begin{equation}\label{def:P}
P(z,\del) = \sum_{|\alpha|\leq m} a_\alpha(z)\left(\pd{}{z}\right)^\alpha
\end{equation}
with coefficients $a_\alpha(z)$ holomorphic in a neighborhood of $p$.
Our goal will be to determine conditions on $\Omega$ and $P$ that guarantee the existence of a solution $u$ to $Pu = 0$ which is holomorphic in $\Omega$ near $p$, but cannot be prolonged across $p$. That is, for some neighborhood $U$ of $p$, $u$ is holomorphic in $U \cap \Omega$, but there is no function $H(z)$ such that $H$ is holomorphic in some neighborhood $W$ of $p$ and $h(z) = H(z)$ for $z \in \Omega\cap W$. As this question is local in nature, we will often be somewhat loose with the neighborhoods of $p$, and will often just say ``near $p$'' to mean in some small enough neighborhood of $p$, or in $\Omega \cap U$ for some small enough neighborhood $U$.

We assume the domain $\Omega$ is pseudoconvex near $p\in \partial\Omega$, since otherwise all holomorphic functions in $\Omega$ near $p$, not just solutions to a holomorphic PDO, would extend generically across $\partial\Omega$. Moreover, we shall assume that $\partial \Omega$ is characteristic at $p$.
Recall that if $\del \Omega$ has a defining equation $\rho(z,\bar z)=0$ (meaning also that the holomorphic gradient $\del \rho\neq 0$ on $\partial \Omega$), then $\del\Omega$ is {\it characteristic} at $p$ with respect to $P$ if $$P_m(p,\del \rho(p,\bar p)) = 0$$
where $\del\rho(z,\bar z) = (\pd{\rho}{z_1}(z,\bar z),\ldots, \pd{\rho}{z_n}(z,\bar z))$ is the holomorphic gradient and  $$P_m(z,\zeta) = \sum_{|\alpha| = m} a_\alpha(z) \zeta^\alpha$$ is the principal symbol of $P$.
As the principal symbol $P_m$ can be viewed as a holomorphic function on the cotangent space of $\bC^n$ near $p$ and the property of being characteristic at $p$ can formulated as saying that the principal symbol vanishes on the conormal space to $\partial\Omega$ at $p$, this definition is independent of the choice of defining function $\rho$ and the choice of local holomorphic coordinates $z$ used near $p$. We note that in some literature such points are called Zerner characteristic, but if there is no fear of confusion, we shall simply call them characteristic. We shall also often omit the phrase ``with respect to $P$'', as there will only be one operator in question. If every point on a hypersurface, either real or complex, is characteristic we say that the hypersurface is everywhere characteristic.

A theorem of Zerner \cite{Zerner} states that if $\del\Omega$ is not characteristic at $p$, then every function $u(z)$ which is holomorphic in $\Omega \cap U$ and satisfies $Pu = 0$ in $U$ can be prolonged holomorphically across $p$. Thus we restrict ourselves to solutions near characteristic points. Further, we will make the simplifying assumptions that $\Omega$ is strictly pseudoconvex and that $\del\Omega$ is {\it simply characteristic} at $p$. That is, $\del\Omega$ is characteristic at $p$ and
$$\del_{\zeta} P_m(p,\del \rho(p,\bar p)) \neq (0,\ldots,0).$$
where $\del_{\zeta}P_m(z,\zeta) = (\pd{P_m}{\zeta_1}(z,\zeta),\ldots,\pd{P_m}{\zeta_n}(z,\zeta))$ is the gradient of $P_m$ in the $\zeta$ variables only.

\subsection{Strong P-Convexity}
The primary tool we will use to construct our solution $u$ is given by the following theorem, proven by Tsuno in \cite{Tsuno}.

\begin{theorem}[Tsuno]\label{Tsuno1}
Let $P(z,\del) = \sum_{|\alpha| \leq m} a_\alpha(z)(\pd{}{z})^\alpha$ with $a_\alpha$ holomorphic in a neighborhood $U$ of zero, and let $S$ be a complex analytic hypersurface through $0$ with defining function $f(z)$ in $U$. Assume that $S$ is everywhere simply characteristic with respect to $P$. Further, assume that
\begin{equation}\label{TsunoAssumption}
P_m(z,\del f(z)) = 0 \mbox{ for all }z\in U.
\end{equation}
Then there exist solutions $u(z)$ to the equation $Pu = 0$ in a neighborhood of zero (possibly multiple valued, ramified along $S$) of the form
$$u(z) = \frac{F(z)}{f(z)} + G(z)\log f(z) + H(z),$$ where $F(z)$, $G(z)$, and $H(z)$ are holomorphic near $0$. Moreover, there exist such solutions for which not both $F(z)$ and $G(z)$ are identically 0 and, hence, that do not extend holomorphically to a neighborhood of $0$.
\end{theorem}

The somewhat artificial assumption in equation (\ref{TsunoAssumption}) can be removed, however. That is, we need only require $S$ to be everywhere characteristic:

\begin{proposition}\label{Tsuno2}
Let $P(z,\del) = \sum_{|\alpha| \leq m} a_\alpha(z)(\pd{}{z})^\alpha$ with $a_\alpha$ holomorphic in a neighborhood $U$ of zero, and let $S$ be a complex analytic hypersurface through $0$ with defining function $f(z)$ in $U$. Assume that $S$ is everywhere simply characteristic with respect to $P$.
Then there exist solutions $u(z)$ to the equation $Pu = 0$ in a neighborhood of zero (possibly multiple valued, ramified along $S$) of the form
\begin{equation}\label{uform}
u(z) = \frac{F(z)}{f(z)} + G(z)\log f(z) + H(z),
 \end{equation}
 where $F(z)$, $G(z)$, and $H(z)$ are holomorphic near $0$. Moreover, there exist such solutions for which not both $F(z)$ and $G(z)$ are identically 0 and, hence, that do not extend holomorphically to a neighborhood of $0$.
\end{proposition}

\begin{proof}
Since $S$ is (everywhere) simply characteristic with respect to $P$ near $0$, we may assume without loss of generality that $\del_{\zeta_n} P_m(z,\del f(z)) \neq 0$. We consider the following, first-order Cauchy problem for a holomorphic function $c(z)$ near $0$:
\begin{equation}\label{Cauchy1}
\left\{ \begin{split}
P_m(z,\del(c(z)f(z)) &= 0 \\
c(z_1,\ldots,z_{n-1},0) &= 1 
\end{split}
\right.
\end{equation}
We claim that the classical Cauchy-Kowalevsky theorem implies that there is a unique holomorphic solution $c(z)$ to \eqref{Cauchy1} near $0$. To see this,
note that we can expand $P_m(z,\del(c(z)f(z)))=P_m(z,f(z)\del c(z)+c(z)\del f(z))$ in a (finite) Taylor series in $\zeta$ at $\zeta=c(z)\del f(z)$ as follows:
\begin{multline}
P_m(z,\del(c(z)f(z)))=
P_m(z,c(z)\del f(z))+\\ f(z)\del_\zeta P_m(z,c(z)\del f(z))\cdot (\del c(z))+ \sum_{2\leq|\beta|\leq m} \frac{(\del_\zeta)^\beta P_m(z,c(z)\del f(z))}{\beta!}f(z)^{|\beta|}(\del c(z))^\beta,
\end{multline}
where $\cdot$ denotes the scalar product $a\cdot b:=\sum_{k=1}^na_kb_k$. Now, since $S$ is everywhere characteristic with respect to $P$,  $f(z)$ is a defining function for $S$, and $P_m(z,\zeta)$ is homogeneous of degree $m$ with respect to $\zeta$, it follows that $P_m(z,c(z)\del f(z))=c(z)^mP_m(z,\del f(z))$ is divisible by $f$, i.e.\
$$
P_m(z,c(z)\del f(z))=c(z)^mA(z)f(z)
$$
for some holomorphic function $A(z)$. Thus, the PDE in \eqref{Cauchy1} can be rewritten, after dividing out by a factor of $f(z)$,
\begin{multline}
A(z)c(z)^m+(c(z))^{m-1}\del_\zeta P_m(z,\del f(z))\cdot (\del c(z))+ \\ \sum_{2\leq|\beta|\leq m} \frac{(\del_\zeta)^\beta P_m(z,\del f(z))}{\beta!}f(z)^{|\beta|-1}(c(z))^{m-|\beta|}(\del c(z))^\beta =0.
\end{multline}
Consider the holomorphic function
\begin{multline}
B(z,w,\zeta):=A(z)w^m+w^{m-1}\del_\zeta P_m(z,\del f(z))\cdot \zeta+ \\ \sum_{2\leq|\beta|\leq m} \frac{(\del_\zeta)^\beta P_m(z,\del f(z))}{\beta!}f(z)^{|\beta|-1}w^{m-|\beta|}\zeta^\beta,
\end{multline}
near the point $(z,w,\zeta)=(0,1,\zeta^0)\in \bC^n\times \bC\times \bC^n$, where $\zeta^0$ is a point such that $B(0,1,\zeta^0)=0$. Note that $\del_{\zeta_n}B(0,1,\zeta^0)=\del_{\zeta_n}P_m(0,\del f(0))\neq 0$ (which also implies that the existence of a solution $\zeta^0$ to the polynomial equation $B(0,1,\zeta^0)=0$). Hence, by the implicit function theorem, we can solve for $\zeta_n$ in the equation $B(z,w,\zeta)=0$ near $(0,1,\zeta^0)$ to obtain $\zeta_n=D(z,w,\zeta')$, where $\zeta'=(\zeta_1,\ldots,\zeta_{n-1})$. Thus, the PDE in \eqref{Cauchy1} can be rewritten
$$
\del_{\zeta_n}c(z)=D(z,c(z),(\del c(z))').
$$
The classical Cauchy-Kowalevsky theorem now implies that \eqref{Cauchy1} has a unique solution $c(z)$ near $0$, as claimed above.

To complete the proof of Proposition \ref{Tsuno2}, we observe that $\psi(z)=c(z)f(z)$ is a defining function for $S$ and $\del_\zeta P_m(0,\del\psi(0)) \neq 0$, so Theorem \ref{Tsuno1} yields the required conclusion.
\end{proof}

We now define $\Omega$ to be {\it strongly $P$-convex} at a point $p \in \del\Omega$ if there is an everywhere characteristic, complex analytic hypersurface $S$ in some open neighborhood $U$ of $p$ such that $S \cap \overline{\Omega} \cap U= \{p\}$. The utility of strong $P$-convexity is seen in the following result.

\begin{corollary}\label{TsunoTheorem}
Let $\Omega \subset \C^n$ be a domain, $p \in \del\Omega$, and let $P(z,\del)$ be a linear differential operator with coefficients holomorphic at $p$ such that $\del\Omega$ is simply characteristic at $p$. If $\Omega$ is strongly $P$-convex at $p$, then there is a solution $u$ to the equation $P(z,\del)u(z) = 0$ which is holomorphic in $\Omega$ near $p$ but does not extend holomorphically to a full neighborhood of $p$.
\end{corollary}

The proof is a direct consequence of Proposition \ref{Tsuno2} and the following facts (details are left to the reader): (i) The everywhere characteristic, complex analytic hypersurface $S$, whose existence is posited by the assumption of strong $P$-convexity at $p$, is necessarily simply characteristic, since $\del \Omega$ is simply characteristic at $p$. (ii) By choosing a sufficiently small, convex neighborhood $U$ of $p$, the intersection $\Omega\cap U$ will be simply connected and, hence, one may choose a holomorphic branch of the $\log$ appearing in \eqref{uform}.

We would like to point out that in order to find a local solution to $Pu=0$ in $\Omega$ that does not extend holomorphically across $p\in\del \Omega$, using Proposition \ref{Tsuno2} as above, it clearly suffices to produce an everywhere characteristic analytic hypersurface $S$ through $p$ which does not enter into $\Omega$, i.e.\ $S\cap\overline{\Omega}\subset\del\Omega$, a weaker condition than that of strong $P$ convexity as defined above. However, in all results in this paper the hypersurface $S$ produced will in fact satisfy $S\cap\overline{\Omega}=\{p\}$ and, therefore we shall take this as our definition of strong $P$-convexity.

\subsection{The problem as a uniqueness problem}

The problem of finding a solution to $Pu = 0$ which cannot be holomorphically extended across $p$ can be considered as one regarding uniqueness, or injectivity of a linear operator. To discuss this we need some further notation.

We use $\mathcal{O}_p$ and $\mathcal{O}_p(\Omega)$ to denote the germs at $p$ of functions holomorphic either in a full neighborhood of $p$ or in $\Omega$ near $p$, respectively. More precisely, $\mathcal{O}_p$ is the set of functions $u$ which are holomorphic in a neighborhood of $p$, with two such functions identified if they agree on some such neighborhood. $\mathcal{O}_p(\Omega)$ is the set of functions $u$ for which there is a neighborhood $U$ of $p$ such that $u$ is holomorphic in $\Omega \cap U$, again identifying two functions that agree in $\Omega \cap W$ for some neighborhood $W$ of $p$.

Note that the linear, holomorphic PDO $P(z,\del)$ given by \eqref{def:P} induces linear operators $P\colon \mathcal O_p\to \mathcal O_p$ and $P\colon\mathcal O_p(\Omega)\to \mathcal O_p(\Omega)$. Since the principal symbol does not vanish on the cotangent space over $p$ for the class of PDOs considered here, the linear map $P\colon\mathcal O_p\to \mathcal O_p$ is surjective. The question whether $P\colon\mathcal O_p(\Omega)\to \mathcal O_p(\Omega)$ is surjective is more delicate and closely related to the interplay between the geometry of $\del\Omega$ and $P$. This latter surjectivity question has been much studied in the literature, at least in the strictly pseudoconvex case (see the next section for a brief history). We also note that, since we have a natural injection $\mathcal O_p\to \mathcal O_p(\Omega)$, the map $P\colon\mathcal O_p(\Omega)\to \mathcal O_p(\Omega)$ descends to the quotient space
\begin{equation}\label{linop}
P\colon \mathcal{O}_p(\Omega) / \mathcal{O}_p \to \mathcal{O}_p(\Omega) / \mathcal{O}_p.
\end{equation}
Since $P\colon \mathcal O_p\to \mathcal O_p$ is surjective, it follows that $P\colon\mathcal O_p(\Omega)\to \mathcal O_p(\Omega)$ is surjective if and only if
\eqref{linop} is. Now, injectivity of \eqref{linop} is equivalent to the statement that every solution of $Pu=f$, for $f\in \mathcal O_p$, is in $\mathcal O_p$, i.e.\ if $Pu$ extends holomorphically across $p$, then so does $u$. Thus, our problem in this paper can be formulated as finding conditions on $P$ and $\del\Omega$ near $p$ such that \eqref{linop} is {\it not} injective. We mention here that there is also a large literature on conditions that imply that \eqref{linop} is injective, e.g.\ \cite{Zerner}, \cite{Tsuno}, \cite{Tsuno80}, \cite{P81}, \cite{EKS98} and references therein.

With this, we are ready to address the problem at hand.

\section{Sufficient Conditions for Strong $P$-Convexity}\label{StrictlyPseudoconvex}

In this section we assume $\Omega$ is strictly pseudoconvex at $p \in \del \Omega$ and that $\del \Omega$ is simply characteristic at $p$ with respect to an operator $P$. The existence of solutions to $Pu=f$ has been studied extensively under these conditions, primarily using the tools of microlocal analysis and hyperfunctions given by Sato, Kawai, and Kashiwara in \cite{SKK}. Using these, Kashiwara and Kawai \cite{KK} derive a Hermitian form $Q$, given below, such that if $Q$ is positive definite on $T_p^{1,0}M \times \C$ then $P: \mathcal{O}_p(\Omega) / \mathcal{O}_p \to \mathcal{O}_p(\Omega) / \mathcal{O}_p$ is surjective, and if $P$ is surjective, then $Q$ is positive semidefinite.

Building on these results, Tr\'{e}preau \cite{Trepreau} (see also \cite{Hormander}) has shown that the surjectivity of $P$ is equivalent to condition $(\Psi)$, as given by Nirenburg and Treves \cite{NT}. $(\Psi)$ is essentially a positivity condition concerning the bicharacteristic strips of $P$ emanating from characteristic points of $M$. Kawai and Takei \cite{KT} later give another equivalent condition, concerning local projections of bicharacteristics.

Our results concern the injectivity of $P:\mathcal{O}_p(\Omega) / \mathcal{O}_p \to \mathcal{O}_p(\Omega) / \mathcal{O}_p$. Rather than using microlocal analysis, we will find sufficient conditions under which we can construct an everywhere characteristic hypersurface that does not intersect $\overline{\Omega}$, and use Proposition \ref{Tsuno2} to prove injectivity.


As in Kawai and Takei \cite{KT}, for each $z \in \C^{n+1}$ near $p$ we define the Hermitian form
\begin{equation}
Q_z(t,\bar t) = \sum_{0\leq j,k\leq n+1} q_{j,k}(z,\bar z)t_j\bar t_k
\end{equation}
for $t=(t_0,\ldots, t_{n+1}) \in \C^{n+2}$, where the coefficients $q_{j,k}$ are given by
\begin{align*}
q_{j,k}(z,\bar z) &= \rho_{z_j\bar z_k}(z,\bar z) & 1\leq j,k \leq n+1 \\
q_{j,0}(z, \bar z) &= P_{m,j}(z,\del \rho(z,\bar z)) + \sum_{k=1}^{n+1} P_m^{(k)}(z,\del \rho(z,\bar z))\rho_{z_j z_k}(z,\bar z) & 1\leq j \leq n+1\\
q_{0,j}(z,\bar z) &= \overline{q_{j,0}(z,\bar z)} & 1\leq j \leq n+1 \\
q_{0,0}(z,\bar z)  &= \sum_{1\leq j,k \leq n+1} P_m^{(j)}(z,\del\rho(z,\bar z)) \overline{P_m^{(k)}(z,\del\rho(z,\bar z))}\rho_{z_j \bar z_k}(z,\bar z)
\end{align*}

Here, as throughout this section, we use subscripts and superscripts on $P_m$ to denote partial derivatives:
\[ P_{m,j}(z,\zeta) = \pd{P_m}{z_j}(z,\zeta),\ \ P_m^{(j)}(z,\zeta) = \pd{P_m}{\zeta_j}(z,\zeta). \]
We will say \emph{condition {\rm (Pos)} holds at $p$} if $Q_p(t,\bar t)$ is positive definite on $\bC\times T_p^{1,0}M $. That is, if
\begin{equation}\label{pos}
\sum_{j,k=0}^n q_{j,k}(p,\bar p)t_j\bar t_k > 0 \mbox{ \qquad when \qquad } t \neq 0 \,\mbox{ and } \ \sum_{j=1}^{n+1}\rho_{z_j}(p,\bar p)t_j = 0.
\end{equation}
We now prove the main result of the section, showing that if (Pos) holds, then $\Omega$ is strongly $P$-convex. Thus we have that at a simply characteristic point where (Pos) holds, $Pu = f$ has a solution for every $f \in \mathcal{O}_p(\Omega) / \mathcal{O}_p$ by \cite{KK}, but such $u$ is not unique in $\mathcal{O}_p(\Omega) / \mathcal{O}_p$ by Proposition \ref{Tsuno2}.

\begin{theorem}\label{StrictlyPseudoconvexTheorem}
Let $\Omega \subset \C^{n+1}$ be a strictly pseudoconvex domain and let $p \in \del\Omega$. Let $P(z, \del)$ be a linear, partial differential operator with coefficients holomorphic near $p$, and assume $\del\Omega$ is simply characteristic with respect to $P$ at $p$. If {\rm (Pos)} holds at $p$, then $\Omega$ is strongly $P$-convex at $p$. 
\end{theorem}

\begin{proof}
As shown in \cite{Hormander}, Corollary 7.4.9, condition (Pos) is independent of the choice of coordinates $z$ near $p$, so we may assume $p = 0$ and choose normal coordinates $z=(z',z_{n+1}) \in \C^n\times \bC$ for $M = \del \Omega$, as given in \cite{BER}. That is, $$\Omega = \{ z\in \C^{n+1} | \ \rho(z,\bar{z}) < 0 \},$$
where
\begin{equation}\label{defRho}
\rho(z, \bar{z}) = -2\real z_{n+1} + \sum_1^n |z_j|^2 + O(3),
\end{equation}
and we write the symbol of $P$, for $\zeta=(\chi,\tau) \in \C^n\times \bC$, as
\begin{equation}\label{Psymb}
P_m(z, (\chi,\tau)) = \tilde a (z)\tau^m + \big( \sum_1^n \tilde b_j (z)\chi_j \big) \tau^{m-1} + O(\|\chi\|^2).
\end{equation}
Here $O(3)$ in \eqref{defRho} represents terms that are total degree $3$ or higher in the power series, and $O(\|\chi\|^2)$ in \eqref{Psymb} denotes terms that contain $\chi_j\chi_k$ for $j,k = 1, \ldots, n$.
Then $0$ is a characteristic point of $M$ with respect to $P$ exactly when $\tilde a(0) = 0$, and $M$ is simply characteristic when $\tilde b_j(0) \neq 0$ for some $j$. Since the vector $(\tilde b_1(0), \ldots, \tilde b_n(0))$ is nonzero, by a unitary linear change in the $z'$ variables (which leaves the form \eqref{defRho} of $\rho$ invariant), we may further assume that $\tilde b_n(0) \neq 0$ and $\tilde b_j(0) = 0$ for $j = 1, \ldots, n-1$. Now, if $c(z)$ is holomorphic and nonzero near $0$, and the equation $cPu = 0$ has a solution $u$ which is holomorphic in $\Omega$ but not at $0$, then of course $u$ also solves $Pu =0$. Thus we may divide $P$ by $\tilde b_n(z)$, and without loss of generality we may assume that the symbol of $P$ has the following form
\begin{equation}\label{p_m}
P_m = a(z)\tau^m + \chi_n + \big( \sum_1^{n-1} b_j(z)\chi_j \big) \tau^{m-1} +O(\|\chi\|^2),
\end{equation}
where $a(0) = b_j(0) = 0$ for $j=1, \ldots, n-1$. Then, $Q_0(t, \bar t)$ is given by

\begin{align*}
q_{j,k}(0) &= \delta_j^k, & j=1,\ldots,n \\
q_{j,n+1}(0) &= \overline{q_{n+1,j}(0)} = 0,  & j=1,\ldots,n+1 \\
q_{j,0}(0) &= \overline{q_{0,j}(0)} = (-1)^m a_{z_j}(0), & \qquad j=1,\ldots,n+1\\
q_{0,0}(0) &= 1.
\end{align*}
Furthermore, $t=(t_0,\ldots,t_{n+1})\in \bC^{n+2}$ belongs to $\bC\times T_0^{1,0}M $ precisely when $t_{n+1}=0$. Thus, (Pos) holds at $p=0$ exactly when the form
\[Q'_0(t,\bar t) = \sum_0^{n} q_{j,k}(0) t_j \bar t_k \]
is positive definite on $\C^{n+1}$. By a direct calculation, one can verify that $1$ is an eigenvalue of multiplicity $(n-1)$ of $Q_0'$, with eigenspace spanned by $\{ (a_{z_j}(0)\evec_2 - a_{z_1}(0)\evec_{j+1} \, | \, j = 2,\ldots,n\}$ (here $\evec_j$ is the $j^{th}$ standard basis vector). The other two eigenvalues are $1\pm \left(\sum_1^n |a_{z_j}(0)|^2\right)^{1/2}$, with corresponding eigenvectors $\left(\pm \left(\sum_1^n |a_{z_j}(0)|^2\right)^{1/2}, \overline{a_{z_1}(0)},\ldots,\overline{a_{z_n}(0)}\right)^T$. Thus in these coordinates (Pos) is equivalent to
\begin{equation}\label{positivity}
\sum_1^n|a_{z_j}(0)|^2 < 1.
\end{equation}

To show that under (Pos), $\Omega$ is strongly $P$-convex, we shall construct an analytic, everywhere characteristic hypersurface $S$ in $\C^{n+1}$ which, under inequality (\ref{positivity}), intersects $\overline{\Omega}$ only at $0$. Clearly, $S$ will have to be tangent to $M=\partial \Omega$ and hence the fact that $M$ is simply characteristic at $0$ implies that if $S$ is everywhere characteristic, then it is simply characteristic near $0$.
We shall seek $S$ of the form $\{z\colon z_{n+1}=f(z')\}$, where as above $z'=(z_1,\ldots, z_n)$ and $f(z')$ is a holomorphic function near $0$ in $\bC^n$  such that $f(0)=0$ and $f_{z'}(0)=0$. The statement that $S$ is everywhere characteristic near $0$ is equivalent to
\begin{equation}\label{echar}
P_m((z',f(z')),(f_{z'}(z'),-1))=0.
\end{equation}
By equation (\ref{p_m}) we have $$P_m(0,(0,\ldots,0,-1))=(-1)^ma(0) = 0,\quad \pd{P_m}{\chi_n}(0,(0,\ldots,0,-1)) = 1.$$ Thus, by the implicit function theorem, there is a holomorphic function $F(z,\chi_1,\ldots,\chi_{n-1})$ such that $P_m(z,(\chi,-1)) = 0$ exactly when $\chi_n = F(z,\chi_1,\ldots,\chi_{n-1})$. Thus, the everywhere characteristic condition \eqref{echar} is equivalent to the PDE $$f_{z_n}(z') = F(z',f(z'),f_{z_1}(z'),\ldots,f_{z_{n-1}}(z')).$$
Let us consider the Cauchy problem
\begin{equation}\label{CauchyProblem}
\left\{ \begin{split}
f_{z_n}(z') &= F(z',f(z'),f_{z_1}(z'),\ldots,f_{z_{n-1}}(z')) \\
f(z_1,\ldots,z_{n-1},0) &= g(z_1,\ldots,z_{n-1}) \\
\end{split} \right.
\end{equation}
where $g$ is a holomorphic function to be determined. This has a unique holomorphic solution $f(z')$ by the Cauchy-Kowalevski theorem. If $g(0)=g_{z_j}(0)=0$ for $j=1,\ldots n-1$, then clearly $f(0) = f_{z_j}(0) = 0$ for $j = 1,\ldots,n-1$. Thus, for each holomorphic data function $g$ with $g(0)=g_{z_j}(0)=0$, we obtain an everywhere, simply characteristic hypersurface $S=\{z\colon z_{n+1}=f(z')\}$ passing through $0$ and tangent to $M$. It now remains to show that we can choose $g$ such that  $S$ stays outside of $\Omega$. Differentiating equation (\ref{echar}) with respect to $z_j$ and evaluating at $0$ gives the Taylor expansion
\begin{equation}
f(z') = \sum_{j=1}^{n-1}a_{z_j}(0)z_jz_n + \frac{1}{2}a_{z_n}(0)z_n^2 + g(z_1,\ldots,z_{n-1}) + O(3).
\end{equation}

The proof of the theorem now follows from the following lemma.
\begin{lemma}\label{BasicInequality}
Let $a_1,\ldots,a_n \in \C$ such that
\begin{equation}\label{a_jInequality}
\sum_{j=1}^n |a_j|^2 < 1.
\end{equation}
Then there is a holomorphic quadratic polynomial $q(z_1,\ldots,z_{n-1})$ such that for any holomorphic function $g(z_1,\ldots,z_{n-1})=q(z_1,\ldots,z_{n-1})+O(3)$, we have
\begin{equation}
\sum_{j=1}^n |z_j|^2 -2\emph{Re } \bigg( \sum_{j=1}^{n-1} a_jz_jz_n + \frac{1}{2}a_n z_n^2 + g(z_1,\ldots,z_{n-1}) \bigg) > 0
\end{equation}
for sufficiently small $z \neq 0$.
\end{lemma}
\begin{proof}
By multiplying $z_1$ by a unimodular constant and then doing the same to each $z_j$, we may assume without loss of generality that $a_j \geq 0$ for $j=1,\ldots, n$. In this case, we let $q(z_1,\ldots,z_{n-1}) = -\frac{1}{2}a_n\sum_{j=1}^{n-1}z_j^2$. Consider the real quadratic form
\[ H(z,\bar z) = \sum_{j=1}^n |z_j|^2 -2\real \bigg( \sum_{j=1}^{n-1} a_jz_jz_n + \frac{1}{2}a_n z_n^2 + q(z_1,\ldots,z_{n-1}). \bigg)\]
To complete the lemma, we show that equation (\ref{a_jInequality}) implies that all the eigenvalues of $H$ are positive. As $H$ is real-valued, we identify $\C^n$ with $\R^{2n}$ via $(x_1,\ldots,x_n,y_1,\ldots,y_n) \cong (x_1 + iy_1,\ldots,x_n + i y_n)$. To simplify notation we let $\gamma = \sqrt{\sum_{j=1}^n a_j^2}$, and let $\evec_j$ denote the $j^{th}$ standard basis vector in $\R^{2n}$.
Then we have
\[ H(x,y) = \sumj \left[ (1+ a_n)x_j^2 + (1-a_n)y_j^2 - 2a_jx_jx_n + 2a_j y_jy_n \right] + (1-a_n)x_n^2 + (1+a_n)y_n^2.\]
A direct calculation verifies that the eigenvalues of $H$ are $1+\gamma$ and $1-\gamma$, each with multiplicity $2$, with eigenspaces spanned by
\begin{align*}
\left\{ \sum_{j=1}^{n-1} a_j\evec_{j+n} + (a_n + \gamma)\evec_{2n},\  \sum_{j=1}^{n-1} a_j\evec_j + (\gamma - a_n)\evec_n \right\}& \mbox{ and }\\
\left\{ \sum_{j=1}^{n-1} a_j\evec_{j+n} + (a_n-\gamma)\evec_{2n},\  \sum_{j=1}^n a_j \evec_j - (a_n + \gamma)\evec_n \right\}&
\end{align*}
respectively. If $n \geq 3$, then the other eigenvalues are $1-a_n$ and $1+a_n$, each of multiplicity $n-2$. Their eigenspaces are spanned by $\{ a_j \evec_{n+1} - a_1 \evec_{n+j} \,|\, j = 2,\ldots, n-1\}$ and $\{a_j\evec_1 -a_1 \evec_j \,|\, j=2,\ldots,n-1\}$, respectively. Thus we see that the smallest eigenvalue of $H$ is $1-\gamma$, and by (\ref{a_jInequality}), $H$ is positive definite. This proves the lemma.
\end{proof}

To complete the proof of Theorem \ref{StrictlyPseudoconvexTheorem}, we choose $g(z_1,\ldots,z_{n-1})$ of the form given by lemma (\ref{BasicInequality}) as initial data in the Cauchy problem (\ref{CauchyProblem}). It follows that
\begin{equation}
\rho((z',f(z')),(\bar z',\overline{f(z')})) > 0
\end{equation}
for $|z|\neq 0$ sufficiently small, which completes the proof.
\end{proof}

Combining this result with Corollary \ref{TsunoTheorem}, we have the following corollary.
\begin{corollary}
If {\rm (Pos)} holds at $p$, there is a solution $u$ to $Pu=0$ which is holomorphic in $\Omega$ near $p$ and cannot be prolonged across $p$.
\end{corollary}


We remark that Tsuno has shown (\cite{Tsuno}, Theorem 3) the existence of a solution to $Pu=0$ in $\Omega$ near $p$ which cannot be prolonged across $p$, but under considerably stronger assumptions than (Pos). In the coordinates given in the proof of Theorem \ref{StrictlyPseudoconvexTheorem} above, his conditions reduce to assuming $a_{z_j}(0) = 0$, for $j=0,\dots,n-1$, and $|a_{z_n}| < 1$. This is certainly much stronger than $\sum_{j=1}^n |a_{z_j}(0)|^2 < 1$ as in Theorem \ref{StrictlyPseudoconvexTheorem}.


\nocite{Folland, Khavinson, NT2}
\bibliographystyle{alpha}
\bibliography{StrongPConvexity}

\end{document}